\documentclass[12pt]{amsart}

\usepackage{amsthm}
\usepackage{amsmath,amssymb}
\usepackage{bm}

\begin{document}

\newtheorem{theorem}{Theorem}
\newtheorem{proposition}[theorem]{Proposition}
\newtheorem{corollary}[theorem]{Corollary}
\newtheorem{lemma}[theorem]{Lemma}
\newtheorem{definition}[theorem]{Definition}
\newtheorem{remark}[theorem]{Remark}
\newtheorem{example}[theorem]{Example}
\numberwithin{theorem}{section}

\newcommand\CC{{\mathcal C}}
\newcommand\LL{{\mathcal L}}
\newcommand\ZZ{{\mathcal Z}}
\newcommand\hL{{\widehat{\mathcal L}}}
\newcommand\Vir{\mathfrak{Vir}}
\newcommand\Cusp{\mathfrak{Cusp}}
\newcommand\Cat{\mathfrak{C}}
\newcommand\Irr{\mathfrak{Irr}}
\newcommand\Z{{\mathbb Z}}
\newcommand\N{{\mathbb N}}
\newcommand\ad{\textup{\,ad\,}}
\newcommand\dv{\textup{div\,}}
\newcommand\R{{\mathcal R}}
\newcommand\C{{\mathbb C}}
\newcommand\dd[1]{\frac{\partial}{\partial #1}}
\newcommand\pd[2]{\frac{\partial #1}{\partial #2}}
\newcommand\p[1]{p(#1)}
\newcommand\A{{\mathcal A}}
\newcommand\V{{\mathcal V}}
\newcommand\FF{{\mathcal F}}
\newcommand\bM{\overline{M}}
\newcommand\hU{\widehat{U}}
\newcommand\AV{{\mathcal {AV}}}
\newcommand\Der{\textup{Der\,}}
\newcommand\Ker{\textup{Ker\,}}
\newcommand\Hom{\textup{Hom}}
\newcommand\Ext{\textup{Ext}}
\newcommand\End{\textup{End}}
\newcommand\Aut{\textup{Aut}}
\newcommand\Span{\textup{Span}}
\newcommand\Image{\textup{Im}}
\newcommand\hM{{\widehat M}}
\newcommand\te{{e}}
\newcommand\tu{{\tilde u}}
\newcommand\barS{{\overline{S}}}
\newcommand\del{{\partial}}
\newcommand\delt{{\partial_t}}
\newcommand\sdel{{{}^*\partial}}
\newcommand\Con{\mathfrak{Con}}
\newcommand\Alg{\mathfrak{Alg}}
\newcommand\Id{\textup{Id}}
\newcommand\Img{\textup{Im\,}}
\newcommand\tC{{\widetilde \CC}}
\newcommand\tL{{\widetilde \LL}}
\newcommand\tdel{{\widetilde \del}}
\newcommand\pp[1]{p(#1)}
\newcommand\br[2]{\langle #1 , #2 \rangle}
\newcommand\cb[2]{\{ #1 , #2 \}}
\newcommand\bcb[2]{\bm{\{} #1 , #2 \bm{\}}}
\newcommand\Ls{{\Lambda^{\circ}}}
\newcommand\tS{{\widetilde S}}
\newcommand\DD{D}
\newcommand\CK{{C{\kern -1pt}K}}

\newcommand*\oline[1]{%
   \vbox{%
     \hrule height 0.6pt
     \kern0.4ex
     \hbox{%
       \kern-0.15em
       \ifmmode#1\else\ensuremath{#1}\fi
       \kern 0.0em
     }
   }
}
\newcommand\bF{\oline{\mathcal F}}

\title
[Locality of superconformal algebras]
{Towards Kac - van de Leur conjecture: locality of superconformal algebras}
\author{Yuly Billig}
\address{School of Mathematics and Statistics, Carleton University, Ottawa, Canada}
\email{billig@math.carleton.ca}

\maketitle

\begin{abstract}
We prove locality of superconformal algebras: every pluperfect superconformal algebra is spanned by coefficients 
of a finite family of mutually local distributions. We also introduce quasi-Poisson algebras and show that they can
be used to construct all known simple superconformal algebras.
\end{abstract}

\section{Introduction}

Superconformal algebras play a crucial role in Conformal Field Theory. These are super generalizations of
the Virasoro algebra, and first examples appearing in physics literature were introduced in 1971 by Andr\'e Neveu and 
John Schwarz \cite{NSz} and Pierre Ramond \cite{R}.

In 1988 Victor Kac and Johan van de Leur proposed a conjectural classification of simple superconformal algebras \cite{KL}.
This classification was amended to include an exceptional simple superconformal algebra $\CK(6)$, which was
subsequently discovered \cite{CK} (see also \cite{S} and \cite{GSL}).

Given a superconformal algebra $\LL$, it is possible to construct new superconformal algebras using diagonalizable 
automorphisms of $\LL$ that fix the Virasoro subalgebra. These are known as twisted forms of $\LL$.
For example, Neveu-Schwarz and Ramond algebras are twisted forms of each other.

Consider a tensor product of an algebra of Laurent polynomials with a Grassmann (exterior) algebra with $N$ odd generators:
$$\R_N = \C[t, t^{-1}] \otimes \Lambda(\xi_1, \ldots, \xi_N).$$
Let $W(1,N)$ be the Lie superalgebra of derivations of this supercommutative algebra:
$$W(1, N) = \Der (\R_N) = \R_N \dd{t} \oplus \sum\limits_{j=1}^N  \R_N \dd{\xi_j} .$$
Algebra $W(1, N)$ is a prototypical example of a simple superconformal algebra, and all algebras in Kac-van de Leur 
classification are twisted forms of various subalgebras in $W(1,N)$.

These include superalgebras of divergence zero vector fields $S(1, N)$ and contact vector fields $K(1, N)$,
thus the notions and ideas of the the theory of superconformal algebras may be traced back to classical
works of Sophus Lie \cite{L} and \'Elie Cartan \cite{C}.

\noindent
{\bf Kac-van de Leur Conjecture.} 
{\it Every simple superconformal algebra is a twisted form of one of the following
algebras: $W(1, N)$, $S^\prime (1, N)$, $\tS(1, 2N)$, $K^\prime(1, N)$, $\CK(6)$.}

Kac and van de Leur specified a conjectural list of the twisted forms that may occur \cite{KL} (see also \cite{MZ}).

Given connections with physics, it is not surprising that all of the algebras on Kac-van de Leur list give rise to 
vertex superalgebras and conformal superalgebras. A major step towards establishing Kac-van de Leur Conjecture
was made by Davide Fattori and Victor Kac \cite{FK}, who classified finite simple conformal superalgebras.

{\bf Theorem.} (\cite{FK}) 
{\it Any finite simple Lie conformal superalgebra is isomorphic to
one of the Lie conformal superalgebras of the following list:
$W_N$, $S_{N, a}$, $\tS_{2N}$, $K_N$ ($N \neq 4$), $K^\prime(4)$, $\CK_6$, and current algebra associated with 
a simple finite-dimensional Lie superalgebra.}

In order to establish Kac-van de Leur Conjecture, it remains to (1) prove that every simple superconformal algebra
has a finite simple conformal superalgebra associated with it; (2) compute the groups of automorphisms
fixing the Virasoro element for conformal superalgebras on Fattori-Kac list; (3) classify the resulting twisted forms of 
superconformal algebras up to an isomorphism.

Victor Kac, Michael Lau and Arturo Pianzola \cite{KLP} computed full automorphism groups for conformal superalgebras 
$K_2$ (a.k.a. $N=2$ conformal superalgebra) and $S_{2, 0}$ (a.k.a. $N=4$ conformal superalgebra). In particular,
they have shown that for $S_{2,0}$ the full group of automorphisms is $SL_2 (\C) \times SL_2 (\C) / \{\pm (I, I)\}$,
while the subgroup preserving the Virasoro element is $SL_2 (\C)$. As a corollary, they established that 
non-isomorphic twisted forms of superconformal algebra $S(1,2)$ are parametrized by one scalar.

In the present paper we prove that pluperfect superconformal algebras (these include simple superconformal algebras and 
their non-split central extensions) belong to the class of graded Lie superalgebras of formal distributions. This means that 
these algebras are spanned by components of a finite family of mutually local distributions. This allows us to associate
a conformal superalgebra to every pluperfect superconformal algebra. However, we can only prove that it is of a finite rank
as an affine conformal superalgebra, and not necessarily a finite conformal superalgebra.

We give an example of a non-pluperfect superconformal algebra which fails to be a graded superalgebra of formal distributions, 
and hence has no conformal superalgebra associated with it. We also give an example of a non-simple superconformal algebra
which is a graded Lie superalgebra of formal distributions, but has no associated finite conformal superalgebra, only an affine one.


Alexander Balinsky and Sergei Novikov \cite{BN} constructed a functor from a class of non-associative algebras, called
Novikov algebras, into the category of superconformal algebras, hoping to construct new simple superconformal 
algebras in this way. Disappointingly for this theory, by the end of the talk that Sergei Novikov gave on this subject
in 1985, Efim Zelmanov \cite{Z} was able to prove that a finite-dimensional simple Novikov algebra has dimension 1
(and corresponds to the Virasoro algebra). In addition to it there is a single two-dimensional simple Novikov superalgebra,
which corresponds to Neveu-Schwarz and Ramond superconformal algebras.

We develop the ideas of Balinsky-Novikov further, introducing quasi-Poisson superalgebras. Just as Poisson superalgebras,
quasi-Poisson superalgebras have a supercommutative multiplication, as well as a Lie bracket, but the relation between these two operations is more intricate. A Poisson superalgebra belongs to the class of quasi-Poisson superalgebras if it admits an even derivation $\DD$ of its supercommutative product with a property that $P = \Id + \DD$ is a derivation of its Lie structure. 
There is a functor from the category of quasi-Poisson superalgebras to the category of superconformal algebras.
One can
associate a simple finite-dimensional quasi-Poisson superalgebra with every superconformal algebra on Kac-van de Leur list.

The paper has the following structure: we discus the category of cuspidal modules for the Virasoro algebra in Section 2 and apply
this theory to superconformal algebras. We establish some technical results on the discrete derivative in Section 3, which we need 
for the proof of mutual locality of generating distributions in pluperfect superconformal algebras in Section 4. In Section 5 we 
discuss a correspondence between twisted forms of superconformal algebras and conformal superalgebras decorated with an automorphism. Finally, we introduce quasi-Poisson superalgebras in Section 6 and list known simple finite-dimensional quasi-Poisson algebras in the Appendix.

\

{\bf Acknowledgements.} This research is supported with a grant from the Natural Sciences and Engineering Research Council of Canada. Parts of this work were carried out during visits to the University of S\~ao Paulo and Institut Henri Poincar\'e. I thank these institutions for the excellent working conditions and hospitality.

\section{Superconformal algebras and cuspidal modules for the Virasoro algebra}

We denote by $\Vir$ the centerless Virasoro algebra with a basis 
$\{ L_n \, | \, n \in \Z \}$ and Lie bracket
$$ [L_n, L_k] = (n - k) L_{n+k} .$$

\begin{definition}
Let $\LL = \LL^0 \oplus \LL^1$ be a Lie superalgebra graded by the additive group $\C$:
$$ \LL^{0,1} = \mathop\oplus\limits_{j \in \C} \LL^{0,1}_j .$$
Lie superalgebra $\LL$ is called a superconformal algebra if the following conditions hold:

\noindent
(S1) $\LL$ contains $\Vir$ as a subalgebra,

\noindent
(S2) The grading on $\LL$ is given by the eigenvalues of $\ad L_0$ and supported on a finite number of $\Z$-cosets in $\C$,

\noindent
(S3) Dimensions of graded components $\LL_j$ are bounded by a constant, independent of $j$.

\end{definition}


We shall denote the parity of $x$ by $\pp{x}$, so that $\pp{x} = 0$ for $x \in \LL^0$ and $\pp{x} = 1$ for $x \in \LL^1$.
Note that in the literature a condition of simplicity is often added in the definition of a superconformal algebra. We find it useful to consider a wider class of algebras.

Let $\LL$ be a Lie superalgebra. A formal series $x(z) = \sum_{n \in \alpha} x_n z^{-n-1}$ with
$\alpha \in \C/\Z$ and $x_n \in \LL$ is called a \emph{distribution}  supported on a coset $\alpha$ with values in $\LL$.

\begin{definition}
\label{Lsfd}
Two distributions $x(z)$, $y(z)$ in Lie superalgebra $\LL$ are called mutually local if
there exists $N \in \N$ such that 
$$(z-w)^N [x(z), y(w)] = 0.$$
\end{definition}

\begin{definition} (\cite{K})
Lie superalgebra $\LL$ is called a Lie superalgebra of formal distributions if it is spanned by the components of a family $\FF$ of mutually local $\LL$-valued distributions.
\end{definition}

 We call $\FF$ a spanning family of $\LL$.

The notion of a Lie superalgebra of formal distributions is rather weak, and in fact every Lie superalgebra satisfies the above definition with a family of distributions $\{ x(z) \, | \, x \in \LL \}$ with $x_n = x$ for all $n \in\Z$. It is easy to see that 
for these distributions $(z-w) [x(z), y(w)] = 0$.

This notion may be strengthened in two ways.  The first of them is:

\begin{definition}
A graded Lie superalgebra $\LL = \mathop\oplus\limits_{j \in \C} \LL_j$ is called a graded Lie superalgebra of formal distributions if it is spanned by the components of a family of mutually local homogeneous distributions
$$ \left\{ x(z) = \sum_{n\in\alpha} x_n z^{-n-1} \right\}$$
with $x_n \in \LL_{n + \delta}$, where $\alpha \in \C/\Z$ and $\delta \in \C$ depend on $x$.
\end{definition}

A second variant of Definition \ref{Lsfd}, the notion of a regular Lie superalgebra of formal distributions will be discussed in Section \ref{dloc}.

A superconformal algebra $\LL$ is perfect, $\LL = \LL^\prime$, if and only if 
$\LL_0 = \mathop\oplus\limits_{j \in \C} [\LL_j, \LL_{-j}]$. Indeed, all components $\LL_j$ with $j \neq 0$ belong to
$\LL^\prime = [\LL, \LL]$ since $\ad (L_0)$ acts on $\LL_j$ as multiplication by $j$.

\begin{definition}
We call a superconformal algebra $\LL$ pluperfect if
$$\LL_0 = \mathop\oplus\limits_{j \in \C \backslash \{ 0 \} } [\LL_j, \LL_{-j}].$$
\end{definition}

\begin{lemma}
(a) The space
$$ \left( \mathop\oplus\limits_{j \in \C \backslash \{ 0 \} } \LL_j \right)
\oplus
\left( \mathop\oplus\limits_{j \in \C \backslash \{ 0 \} } [\LL_j, \LL_{-j}] \right)$$
is an ideal in a graded algebra $\LL$.

(b) Every simple superconformal algebra is pluperfect.

(c) Every non-split central extension of a simple superconformal algebra is pluperfect.
\end{lemma}

The proof of this lemma is obvious and we omit it.

One of the goals of the present paper is to prove the following theorem:

\begin{theorem}
\label{main}
Every pluperfect superconformal algebra is a graded Lie superalgebra of formal distributions with a finite spanning family.
\end{theorem}

The condition for the superconformal algebra to be pluperfect, $\LL = [\LL, \LL]$ in the above theorem is essential. Let us give an example of a superconformal algebra which fails to be a graded Lie superalgebra of formal distributions. 
This example is based on the coadjoint representation of the Virasoro Lie algebra.

\begin{example}
\label{nonlocal}
\normalfont
Let $\LL$ be a Lie algebra with a basis $\{ L_n, F_n, G_n \, | \, n \in \Z \}$ with Lie brackets
\begin{align*}
&[L_n, L_m] = (n-m) L_{n+m}, \\
&[L_n, F_m] = - (m + 2n) F_{n+m}, \\
&[L_n, G_m] = - (m + n) G_{n+m} + \delta_{m,0} n^3 F_n, \\
&[F_n, F_m] = [F_n, G_m] = [G_n, G_m] = 0.
\end{align*}
\end{example}

This Lie algebra is not perfect - element $G_0$ does not belong to $[\LL, \LL]$. The derived subalgebra $\LL^\prime = [\LL, \LL]$ is  a graded Lie superalgebra of formal distributions, and is spanned by mutually 
local distributions
\begin{equation*}
L(z) = \sum_{n\in\Z} L_n z^{-n-2}, \ \ 
F(z) = \sum_{n\in\Z} F_n z^{-n-1}, \ \ 
G(z) = \sum_{n\in\Z} n G_n z^{-n-1}. 
\end{equation*}


Let us recall the results of \cite{M} and \cite{BF} on cuspidal modules for the Virasoro algebra.

\begin{definition}
A module $M$ for the Virasoro algebra $\Vir$ is called cuspidal if $M$ has a weight decomposition
$$M = \mathop\oplus\limits_{\lambda \in\C} M_\lambda,$$
where $M_\lambda = \{ m\in M \, | \, L_0 m = \lambda m \}$,
such that there exists a uniform bound $K$ for the dimensions of the weight spaces,
$\dim M_\lambda \leq K$ for all $\lambda \in\C$ and the grading is supported on a finite number of 
$\Z$-cosets.
\end{definition}

We will denote by $\Cusp$ the category of cuspidal modules for the Virasoro algebra.
The condition on a finite number of cosets in the support of the grading is not essential and is added in order to simplify certain statements.

It follows immediately from the definitions that any superconformal algebra $\LL$ is a cuspidal module with respect to
the adjoint action of $\Vir$.


Important examples of cuspidal modules are the {\it tensor modules} $V(\alpha, \beta)$,
$\alpha \in \C$, $\beta \in\C/\Z$,
which have bases $\{ v_k \, | \, k \in \beta \}$ and the action is
$$ L_n v_k = - (k + \alpha n) v_{n+k} .$$
Tensor modules $V(\alpha, \beta)$ are pairwise non-isomorphic. They are simple, except in two cases, $\alpha = 0, 1$, $\beta = \Z$. Viewing $\Vir$ as the Lie algebra of polynomial vector fields on a circle, $\Vir = \Der \C[t, t^{-1}]$,
the module $V(0, \Z)$ is interpreted as the module of functions on the circle,
$V(0, \Z) \cong \C[t,t^{-1}]$ and $V(1, \Z)$ is the module of 1-forms on the circle.
$V(0, \Z)$ has a unique proper submodule spanned by $v_0$ (constant function), and 
$V(1, \Z)$ has a unique proper submodule spanned by $\{ v_k \, | \, k \neq 0 \}$ (exact 1-forms).
The quotient ${\overline V} = V(0, \Z) / \left< v_0 \right>$ is a simple $\Vir$-module and it is isomorphic to the proper submodule of $V(1, \Z)$. The corresponding isomorphism is induced by the differential map 
$d: \, \C[t,t^{-1}] \rightarrow \C[t,t^{-1}] dt$, where $d(t^k) = k t^{k-1} dt$.


Simple weight modules for $\Vir$ were classified by Mathieu \cite{M}. It follows from his classification that simple objects in the category of cuspidal modules are simple tensor modules, ${\overline V}$ and the trivial 1-dimensional module $\C$.

%
%

Let us discuss the blocks of the category of cuspidal modules. There is an obvious partition of $\Cusp$ with respect to the spectrum of operator $L_0$. 
Indeed, any weight $\Vir$-module decomposes into a direct sum of submodules supported on $\Z$ cosets in $\C$, 
\begin{equation}
\label{dirsum}
M = \mathop\oplus\limits_{\beta \in\C/\Z} M[\beta] , \text{\ where \ }
M[\beta]  =   \mathop\oplus\limits_{\lambda \in \beta} M_\lambda .
\end{equation}

It turns out that this decomposition of $\Cusp$ into a direct sum of subcategories may be further refined. Let us quote the following lemma:

\begin{lemma}
\label{catext}
(\cite{J}, \cite{NS})
Let $\Cat$ be an abelian category in which every object has a finite composition series. Let 
$\Irr = \mathop\cup\limits^{\cdot}_{\gamma \in S} \Irr(\gamma)$ be a partitioning of simple objects in 
$\Cat$ satisfying the property that for $M_1 \in \Irr(\gamma_1)$, $M_2 \in \Irr(\gamma_2)$ with $\gamma_1 \neq \gamma_2$ we have $\Ext^1_\Cat (M_1, M_2) = 0$.
Let $\Cat(\gamma)$ be a subcategory of objects that have all simple factors in $\Irr(\gamma)$.
For $M \in \Cat$ and $\gamma\in S$ define $M[\gamma]$ to be the sum of all subobjects in $M$ which belong to $\Cat(\gamma)$.

 Then for every $M, M^\prime \in \Cat$

(a) $M = \oplus_{\gamma\in S} M[\gamma]$,

(b) $\Hom_\Cat (M, M^\prime) = \oplus_{\gamma\in S} 
\Hom_\Cat (M[\gamma], M^\prime[\gamma])$. 
\end{lemma}

We partition irreducible objects in category $\Cusp$ by the set $S = \C/\Z \times \C/\Z$ with 
$V(\alpha, \beta) \in \Irr(\alpha+\Z, \beta)$ and ${\overline V}, \C \in \Irr(\Z, \Z)$. For $i=1,2$ 
fix $\gamma_i = (\alpha_i, \beta_i)$ with $\alpha_i, \beta_i \in \C/\Z$ and let $M_i \in \Irr(\gamma_i)$.
It follows from (\ref{dirsum}) that $\Ext^1 (M_1, M_2) = 0$ when $\beta_1 \neq \beta_2$. If $\beta_1 = \beta_2$,
it can be seen from the results of \cite{MP}, \cite{CKW}, \cite{D} that $\Ext^1 (M_1, M_2) = 0$ when the cosets 
$\alpha_1, \alpha_2$ are distinct.

By Lemma \ref{catext} we get that a superconformal algebra $\LL$ decomposes as a $\Vir$-module:
\begin{equation*}
\LL = \mathop\oplus\limits_{\alpha, \beta \in \C/\Z}\LL[\alpha, \beta], 
\end{equation*}
where $\LL[\alpha, \beta]$ is the sum of $\Vir$-submodules in $\LL$ with all simple factors of their composition series 
in $\Irr(\alpha, \beta)$.

\begin{theorem}
\label{grading}
Let $\LL$ be a superconformal algebra. Then the decomposition
\begin{equation*}
\LL = \mathop\oplus\limits_{\alpha, \beta \in \C/\Z}\LL[\alpha, \beta], 
\end{equation*}
is a grading of $\LL$ as a Lie superalgebra.
\end{theorem}
\begin{proof}
Iohara-Mathieu \cite{IM} classified equivariant maps of tensor $\Vir$-modules
\begin{equation}
\label{eqmap}
V(\alpha_1, \beta_1) \times V(\alpha_2, \beta_2) \rightarrow V(\alpha_3, \beta_3).
\end{equation}
It follows from their classification that a necessary condition for the existence of a non-zero equivariant map (\ref{eqmap}) is
$\beta_3 = \beta_1 + \beta_2$ and $\alpha_3 \in \alpha_1 + \alpha_2 + \Z$. This implies the claim of the Theorem.
Indeed, if
$\left[ \LL[\alpha_1 + \Z, \beta_1], \LL[\alpha_2 + \Z, \beta_2] \right] \not\subset  
\LL[\alpha_1 + \alpha_2 + \Z, \beta_1 + \beta_2]$,
we get a non-zero equivariant map $\LL[\alpha_1 + \Z, \beta_1] \times \LL[\alpha_2 + \Z, \beta_2] \rightarrow  
\LL[\alpha_3 + \Z, \beta_3]$ with $(\alpha_3 + \Z, \beta_3) \neq (\alpha_1 + \Z, \beta_1)  + (\alpha_2 + \Z, \beta_2)$. Considering composition series filtrations in $\LL[\alpha_i + \Z, \beta_i]$, $i=1,2,3$, 
we can construct a non-zero equivariant map between simple modules in these composition series, contradicting the result 
of \cite{IM}. 
\end{proof}

It follows from the description of simple cuspidal modules that for indecomposable cuspidal modules dimensions of weight spaces are actually constant outside zero weight, and in the zero weight component we may have an abnormal behaviour, such as a gap, as in the case of ${\overline V}$. Below we will present a regularization method developed in \cite{BF} that allows us to fix this irregularity.
 Its idea is to attach to a cuspidal module a new $\Vir$-module which admits an additional action of the algebra of Laurent polynomials. Below we will denote the Virasoro algebra by $\V$ and the commutative algebra of functions by 
$\A = \C[t, t^{-1}]$.

We note that $\A$ is a module for the Lie algebra $\V$ with the action $L_n t^k = - k t^{k+n}$,
and $\V$ is a module for the commutative unital algebra $\A$ with the action 
$t^k L_n = L_{n+k}$.

\begin{definition}
We call $M$ an $\AV$-module if it is a module for the Lie algebra $\V$, a module for commutative unital algebra $\A$ and these two actions are compatible via the Leibniz rule:
$$\eta (f m) = \eta(f) m + f (\eta m),$$
for $\eta \in \V$, $f \in \A$, $m\in M$.
\end{definition} 

It is easy to check that tensor modules are $\AV$-modules. They are all simple in the category of $\AV$-modules.

It is easy to see that a cuspidal $\AV$-module is actually a free $\A$-module of a finite rank, it is freely generated as an $\A$-module by any of its weight spaces. This shows that cuspidal $\AV$-modules do not have abnormal behaviour at the zero weight space.

Let us show how to attach an $\AV$-module to a cuspidal $\Vir$-module.

Let $M$ be a $\Vir$-module. We first note that the coinduced module $\Hom_\C (\A, M)$
is an $\AV$-module with the action
\begin{align*}
&(f \varphi) (g) = \varphi (fg), \\
&(\eta \varphi) (g) = \eta( \varphi (g)) - \varphi(\eta(g)),
\end{align*}
where $\varphi \in \Hom_\C (\A, M)$, $\eta \in \V$, $f, g \in \A$ (see \cite{BF} for details).
The coinduced module is too big, and we wish to construct a smaller $\AV$-module.

\begin{definition}
Let $M$ be a $\Vir$-module. Its $\AV$ cover $\hM$ is a submodule in $\Hom_\C (\A, M)$
spanned by 
$$\{ \psi( \eta, m) \, | \, \eta \in \V, m \in M \}, $$
where
$$ \psi (\eta, m) (g) = (g \eta) m .$$
\end{definition}

\begin{theorem} [\cite{BF}]
\label{cusp}
Let $M$ be a cuspidal $\Vir$-module. Then

(a) There exists $N \in \N$ such that 
$$ 
\sum_{a=0}^N (-1)^a {N \choose a} L_{p+a} L_{q-a} = 0 \text{\ in } M,$$

(b) The cover $\hM$ of $M$ is a cuspidal $\AV$-module,

(c) The map $\pi: \, \hM \rightarrow M$, $\pi(\psi(\eta, m)) = \eta m$ is a homomorphism of $\V$-modules with $\Image (\pi) = \V M$.
\end{theorem}

It is clear from the definition that for $M \in \Cusp[\alpha, \beta]$, its cover $\hM$ belongs to the same block.

 The advantage of using $\AV$-modules is that the structure of cuspidal $\AV$-modules is 
well-understood.

 Denote by $W_+$ a Lie algebra with a basis $\{ \te_n \, | \, n \geq 0 \}$ and Lie
bracket $[\te_n, \te_k] = (k-n) \te_{n+k}$. Even though $W_+$ could be embedded into the Virasoro algebra, we will not view it as a subalgebra in $\Vir$. Let us recall some results from \cite{Bi} on finite-dimensional $W_+$-modules:
\begin{lemma}
\label{Lplus}
Let $U$ be a finite-dimensional $W_+$-module. 

(a) There exists $N \in \N$ such that $\te_n U = 0$ for all $n > N$.

(b) If $U$ is irreducible then it has dimension 1, $U = \C u$ and there exists $\alpha \in \C$ such that $\te_0 u = \alpha u$ and $\te_n u = 0$ for $n \geq 1$.
\end{lemma}

\begin{theorem}
\label{coord}
Let $M$ be a cuspidal $\AV$-module in a block $\Cusp(\alpha, \beta)$, $\alpha, \beta \in \C/\Z$, 
and let $U$ be one of its weight spaces. 
 Then
$$ M \cong  \mathop\oplus\limits_{k \in \beta} t^k \otimes U. $$
The space $U$ admits the structure of a $W_+$-module, so that the action of $\Vir$ is given by
$$L_n (t^k \otimes u) = - k t^{n + k} \otimes u 
- \sum_{s \geq 1} \frac{n^s}{s!} t^{n + k} \otimes \te_{s-1} u,$$
while $\A$ acts on the left tensor factor by multiplication. 
\end{theorem}

Note that by Lemma \ref{Lplus} (a) the sum in the above formula is finite.

\section{Discrete derivative}

One of the tools that we will be using in this paper is the notion of the discrete derivative. 

\begin{definition}
Let $F(p)$ be a function $\alpha \rightarrow R$, where $\alpha$ is a coset of $\Z$ in $\C$ and $R$  is an abelian group. 
The discrete derivative $\Delta_p F(p)$ is a function $\alpha \rightarrow R$ defined as
$$\Delta_p F(p) = F(p) - F(p+1).$$
\end{definition}

We also define a backwards discrete derivative $\Delta^{-}_{p} F(p) = F(p) - F(p-1)$.
Throughout the paper, when we write $F(p) = 0$ we mean that $F$ is a zero function. 

We will also need a two-variable version of the discrete derivative of a function 
$F(p,q): \, \alpha \times \beta \rightarrow R$, defined as 
$$\Delta_{p,q} F(p,q) = F(p,q) - F(p+1,q-1).$$ 

Let us show a relation between discrete derivative and locality:

\begin{lemma}
\label{locder}
Two distributions $x(z) = \sum_{n\in\alpha} x_n z^{-n-1}$ and $y(z) = \sum_{n\in\beta} y_n z^{-n-1}$
are mutually local if and only if there exists $N \in \N$ such that 
$$\Delta^N_{p,q} [x_p, y_q] = 0.$$ 
\end{lemma}

\begin{proof}
Consider the expression from the definition of locality
$$(z-w)^N [x(z), y(w)] = 
\left( \sum_{a=0}^N (-1)^{N-a} {N \choose a} z^a w^{N-a} \right)
\sum_{n,s\in\Z} [x_n, y_s] z^{-n-1} w^{-s-1} .$$
The coefficient at $z^{-p} w^{N-q}$ in the above expression is
$$\sum_{a=0}^N (-1)^{N-a} {N \choose a} [x_{p+a}, y_{q-a}]
= (-1)^N \Delta^N_{p,q}  [x_p, y_q],$$
and the claim of the lemma follows.
\end{proof}

Let us list some elementary properties of discrete derivative which we will need. In parts 
(c), (d) of the following lemma we assume that $R$ is a ring.

\begin{lemma}
\label{properties}
(a) $\Delta^N_p F(p) = 0$ if and only if $F(p)$ is a polynomial function of degree at most $N-1$.

(b) For a function $F(p,q,r): \alpha \times \beta \times \gamma \rightarrow R$,
$\alpha, \beta, \gamma \in \C/\Z$,
$$\Delta_{p,q} \Delta_{q,r} F(p,q,r) = \Delta_{q,r} \Delta_{p,q} F(p,q,r) .$$   

(c) $$\Delta_{p,q} \left( F(p,q) G(p,q) \right) 
=   \left( \Delta_{p,q} F(p,q) \right)  G(p,q) + F(p+1, q-1) \Delta_{p,q} G(p,q). $$

(d) If $\Delta^N_{p,q} F(p,q) = 0$ and $\Delta^K_{p,q} G(p,q) = 0$ then
$$\Delta^{N+K-1}_{p,q} \left( F(p,q) G(p,q) \right) = 0.$$

(e) $$\Delta_{p,q} H(p+s, q) = \Delta_{s, q} H(p+s, q).$$

(f) $$\Delta^N_{p,q} \Delta^K_{s,q} H(p+s, q) = \Delta^{N+K}_{p, q} H(p+s, q).$$
\end{lemma}

The proof of this lemma is straightforward and we omit it.

\begin{lemma}
\label{dec}
Suppose
\begin{equation}
\label{start}
\Delta^N_{p,q} \sum_{l=0}^n p^l H_l (p+s, q) = 0.
\end{equation}
Then for every $l = 0, \ldots, n$ there exists $K \in \N$ such that 
$$\Delta^K_{p,q} p^l H_l (p+s, q) = 0.$$
\end{lemma}
\begin{proof}
Let us prove this Lemma by induction on $n$. If $n=0$, the statement holds trivially.
For the step of induction, apply operator $\Delta_{p,s}^n$ to (\ref{start}). 
Since $\Delta_{p,s} H_l (p+s, q) = 0$ and $\Delta_{p,s} F(p) = \Delta_p F(p)$, we get
$n! \Delta^N_{p,q}  H_n (p+s, q) = 0$. Then by Lemma \ref{properties} (d), we have
\begin{equation*}
\Delta^{N+n}_{p,q}  p^n H_n (p+s, q) = 0.
\end{equation*}
Subtracting this from $\Delta^n_{p,q}$ applied to (\ref{start}), we get
\begin{equation*}
\Delta^{N+n}_{p,q} \sum_{l=0}^{n-1} p^l H_l (p+s, q) = 0,
\end{equation*}
and the claim of the Lemma follows from the induction assumption.
\end{proof}

\section{Locality of distributions in superconformal algebras}
\label{fields}

We begin by applying the machinery described in the previous sections to superconformal algebras.

Let $\LL$ be a superconformal algebra. It is a cuspidal module with respect to the adjoint action of $\Vir$. Let $\hL$ be its $\AV$ cover. By Theorem \ref{coord} we can realize
$\hL$ as 
$$\hL = \mathop\oplus\limits_{\beta \in \C/\Z}
\mathop\oplus\limits_{k\in \beta} t^k \otimes U[\beta] ,$$
where $U[\beta]$ is a weight space in $\Vir$-submodule $\oplus_{\alpha\in\C/\Z} \LL[\alpha, \beta]$.

We are going to construct a basis in a finite-dimensional vector space 
\begin{equation}
\label{U}
U = \oplus_{\beta \in \C/\Z} U[\beta].
\end{equation}
By Lemma \ref{Lplus} (a), $U$ is a finite-dimensional representation of a solvable Lie algebra $W_+ / \left< \te_n \, | \, n > N \right>$. By Lie's theorem, there exists a flag
$$ 0 \subset \FF_1 \subset \FF_2 \subset \ldots \subset \FF_r = U$$ such that
$\te_0 \FF_j \subset \FF_j$ and $\te_n \FF_j \subset \FF_{j-1}$ for $n\geq 1$. Since $\dim \FF_j / \FF_{j-1} = 1$, there exist scalars $\alpha_j \in\C$ such that $\te_0$ act on $ \FF_j / \FF_{j-1}$ 
as multiplication by $\alpha_j$. We are going to choose a basis 
$\{u^1, \ldots, u^r \}$ in $U$, which is compatible with this flag and has an additional 
property that the matrix of $\te_0$ has a block decomposition with respect its generalized 
eigenspaces, i.e.,
$$\te_0 u^i = \alpha_i u^i + \sum\limits_{j > i} c^0_{ij} u^j$$
where $c^0_{ij} = 0$ when $\alpha_i \neq \alpha_j$.
The action of $\te_n$ with $n \geq 1$ will be written as
$$\te_n u^i =  \sum\limits_{j > i} c^n_{ij} u^j.$$
Of course, we can choose such a basis in a way that respects the direct sum decomposition
(\ref{U}), i.e., each $u^i$ belongs to one of the spaces $U[\beta]$, which implies that $c^n_{ij} = 0$ when $u^i$ and $u^j$ belong to different components. We will denote by
$\beta_i$ the coset corresponding to $u_i$.

We get a basis in $\hL$: $\{ u^i_k = t^k \otimes u^i \, | \, i=1,\ldots, r, \, k \in \beta_i \}$. Applying Theorem \ref{coord}, we get the action of $\Vir$ on $\hL$
in this basis:
\begin{equation}
\label{act}
L_n u^i_k = - (k + n \alpha_i) u^i_k 
- \sum_{j > i} \sum_{s > 0} \frac{n^s}{s!} c_{ij}^{s-1} u^j_{n+k},
\end{equation} 
which we can write as
\begin{equation}
\label{act2}
L_n u^i_k = - (k + n \alpha_i) u^i_k 
- \sum_{j > i} F_{ij} (n) u^j_{n+k},
\end{equation} 
where $F_{ij} (n)$ are polynomials in $n$ without constant terms, and also without terms linear in $n$ whenever $\alpha_i \neq \alpha_j$. 

We use the map $\pi: \hL \rightarrow \LL$ to project basis elements into
$\LL$. We set
\begin{equation}
\label{gens}
\tu^i_k = \left\{ 
\begin{matrix}
\pi (u^i_k), \text{\ if } \alpha_i \neq 1, \\
k \pi(u^i_k), \text{\ if } \alpha_i = 1, \\
\end{matrix}
\right.
\end{equation}
and we define distributions in $\LL$, $\tu^i (z) = \sum_{k\in\beta_i} \tu^i_k z^{-k}$.
Clearly, these distributions are homogeneous with respect to the grading on $\LL$ by the eigenvalues of $\ad L_0$.

As we can see from Example \ref{nonlocal}, care needs to be taken when dealing with the distributions corresponding to $\alpha = 1$.  The next proposition is the key technical result.

\begin{proposition}
\label{mutual}
Fields $\{\tu^1 (z), \ldots, \tu^r (z) \}$ are mutually local. 
\end{proposition}
\begin{proof}
By Lemma \ref{locder}, we need to show that 
\begin{equation}
\label{newloc}
\Delta^S_{p,q} [\tu^i_p, \tu^j_q] = 0
\end{equation}
for some $S\in\N$. Our idea is to derive this from the relation 
$\Delta^N_{p,q} \ad L_p \ad L_q = 0$, given by Theorem \ref{cusp} (a).

As an intermediate step, we are going to prove the following identities:

\begin{lemma}
\label{inter}
There exists $M \in \N$ such that for all $i = 1, \ldots, r$,
$$\Delta^M_{p,q} \ad \tu^i_p \ad L_q = 0 .$$
\end{lemma}
\begin{proof}
We shall prove this claim by a descending induction in $i$. Let us assume that the claim holds for all $j = i+1, \ldots, r$. Consider the equality
$$\left[ \Delta^N_{p,q} \ad L_p \ad L_q, \ad \pi (u^i_s) \right] = 0.$$
The left hand side can be rewritten as
\begin{align*}
\Delta^N_{p,q} \ad \pi\left( L_p u^i_s \right) \ad L_q 
+ \Delta^N_{p,q}  \ad L_p  \ad \pi \left( L_q u^i_s \right) . 
\end{align*}
We expand the second summand using (\ref{act2}):
\begin{align*}
\Delta^N_{p,q}& \ad \pi\left( L_p u^i_s \right) \ad L_q 
+ \Delta^N_{p,q}  \ad L_p  \ad \pi \left( L_q u^i_s \right) =\\
\Delta^N_{p,q}& \ad \pi\left( L_p u^i_s \right) \ad L_q \\
&- \Delta^N_{p,q} (s + \alpha_i q) \ad L_p  \ad \pi  u^i_{q+s} 
- \sum_{j>i}  F_{ij} (q) \ad L_p  \ad \pi  u^j_{q+s} .
\end{align*}
Next we want to apply a power of $\Delta_{s,q}$ to this expression. We note that
$\Delta_{s,q} u^j_{q+s} = 0$ and $\Delta_{s,q}  F_{ij} (q) = \Delta^{-}_q F_{ij} (q)$. 
Choose $K \geq 2$ such that it exceeds degrees of all polynomials $F_{ij}$. Then applying
$\Delta^K_{s,q}$ to the previous expression we get
\begin{equation*}
\Delta^K_{s,q} \Delta^N_{p,q} \ad \pi\left( L_p u^i_s \right) \ad L_q  = 0.
\end{equation*} 
Next, expanding the left hand side using (\ref{act2}), we get
\begin{align*}
\Delta^K_{s,q} \Delta^N_{p,q} (s + \alpha_i p) &\ad \pi(u^i_{p+s}) \ad L_q \\
&+  \sum_{j>i}  \Delta^K_{s,q} \Delta^N_{p,q}  F_{ij} (p) \ad \pi(u^j_{p+s}) \ad L_q = 0 .
\end{align*}
By induction assumption for $j >i$ with $\alpha_j \neq 1$ we have
$$\Delta^M_{p,q} \ad \pi(u^j_{p}) \ad L_q = 0$$
and hence
$$\Delta^M_{p,q} \ad \pi(u^j_{p+s}) \ad L_q = 0.$$
If $L$ is the maximum of degrees of all polynomials $F_{ij}$, by Lemma \ref{properties} (d), we get for $j > i$ with $\alpha_j \neq 1$
$$\Delta^{M+L}_{p,q}  F_{ij} (p) \ad \pi(u^j_{p+s}) \ad L_q = 0.$$
We may assume that $N \geq M + L$. Then we get
\begin{align}
\label{fork}
\Delta^K_{s,q} \Delta^N_{p,q} &(s + \alpha_i p) \ad \pi(u^i_{p+s}) \ad L_q\\
&+  {\sum_{j>i}}^\prime  
\Delta^K_{s,q} \Delta^N_{p,q}  F_{ij} (p) \ad \pi(u^j_{p+s}) \ad L_q =0, \nonumber
\end{align}
where $ {\sum\limits_{j>i}}^\prime$ is a summation only over indices $j$ with $\alpha_j = 1$.

We consider now two cases, $\alpha_i = 1$ and $\alpha_i \neq 1$.

Suppose $\alpha_i = 1$. Then we have
\begin{align*}
\Delta^K_{s,q} \Delta^N_{p,q} \ad \tu^i_{p+s} \ad L_q
+  {\sum_{j>i}}^\prime  
\Delta^K_{s,q} \Delta^N_{p,q}  F_{ij} (p) \ad \pi(u^j_{p+s}) \ad L_q =0.
\end{align*}
Now we apply to this expression Lemma \ref{dec} with $l=0$. Since polynomials $F_{ij}(p)$ have no constant terms, we get
\begin{equation*}
\Delta^K_{s,q} \Delta^R_{p,q} \ad \tu^i_{p+s} \ad L_q = 0.
\end{equation*}
By Lemma \ref{properties} (f), this is equivalent to
\begin{equation*}
\Delta^{K+R}_{p,q} \ad \tu^i_{p+s} \ad L_q = 0.
\end{equation*}
Setting $s=0$ we obtain the claim of the Lemma when $\alpha_i =1$. 

Now we treat (\ref{fork}) with $\alpha_i \neq 1$. We rewrite $(s + \alpha_i p)$ as
$(s + p) + (\alpha_i - 1) p$ and apply to (\ref{fork}) Lemma \ref{dec} with $l = 1$.
Recall that in summation over $j$ we only have terms with $\alpha_j = 1$. Since
$\alpha_i \neq \alpha_j$, polynomials $F_{ij} (p)$ do not contain linear terms. As a result
we will get
$$ (\alpha_i - 1) \Delta^K_{s,q} \Delta^R_{p,q} p \ad \pi(u^i_{p+s}) \ad L_q = 0.$$
Applying now $\Delta_{p,s}$ we get
$$\Delta^K_{s,q} \Delta^R_{p,q} \ad \tu^i_{p+s} \ad L_q = 0.$$
Applying Lemma \ref{properties} (f) and setting $s = 0$, we obtain the claim of the Lemma.
\end{proof}

Now let us return to the proof of Proposition \ref{mutual}.
This proof will be very similar to the proof of the previous lemma.
We will prove (\ref{newloc}) by a decreasing induction in $j$.
By Lemma \ref{inter} we have
\begin{equation*}
\Delta^M_{p,q} [\tu^i_p,  [L_q, \pi(u^j_s) ]] = 0,
\end{equation*}
which becomes
\begin{equation}
\label{exp3}
\Delta^M_{p,q} (s + \alpha_j q) [\tu^i_p,  \pi(u^j_{q+s}) ]
+ \sum_{k > j} \Delta^M_{p,q} F_{jk}(q)  [\tu^i_p,  \pi(u^k_{q+s}) ] = 0.
\end{equation}
Using the same argument as in the proof of Lemma \ref{inter} we see that for $k$ with
$\alpha_k \neq 1$ and sufficiently large $M$ we get
$$\Delta^M_{p,q} F_{jk}(q)  [\tu^i_p,  \pi(u^k_{q+s}) ] = 0.$$
Eliminating these terms in (\ref{exp3}), we get
\begin{equation*}
\label{exp2}
\Delta^M_{p,q} (s + \alpha_j q) [\tu^i_p,  \pi(u^j_{q+s}) ]
+ {\sum_{k > j}}^\prime \Delta^M_{p,q} F_{jk}(q)  [\tu^i_p,  \pi(u^k_{q+s}) ] = 0,
\end{equation*}
where the summation in $k$ only contains terms with $\alpha_k = 1$.

Consider now two cases, $\alpha_j = 1$ and $\alpha_j \neq 1$. If $\alpha_j = 1$,
we apply Lemma \ref{dec} with $l = 0$ and obtain
$$\Delta^S_{p,q} [\tu^i_p,  \tu^j_{q+s}] = 0,$$
which yields the required claim after we set $s=0$.

If $\alpha_j \neq 1$, we use  Lemma \ref{dec} with $l = 1$ and obtain
$$(\alpha_j - 1) \Delta^S_{p,q} q [\tu^i_p,  \tu^j_{q+s}] = 0,$$
and we get the desired claim after applying operator $\Delta_{q,s}$ and then setting $s = 0$.

 This completes the proof ot the Proposition.
\end{proof}

Now we are ready to prove Theorem \ref{main}.

\begin{proof}
Let $\LL$ be a pluperfect superconformal algebra. We need to show that it is spanned by the 
coefficients of a family of mutually local distributions. Consider the family of distributions
$S = \{\tu^1 (z), \ldots, \tu^r (z)$ of Proposition \ref{mutual}. The coefficients of these
distributions were constructed from a basis of $\AV$-cover $\hL$. The projection 
$\pi: \, \hL \rightarrow \LL$ has $[\Vir, \LL]$ as its image. Since for every non-zero $j \in \C$, $L_0$ acts on $\LL_j$ as a non-zero scalar, we get that $\LL_j$ with $j\neq 0$
belongs to the image of $\pi$ and hence is spanned by the coefficients at $z^{-j-1}$ of 
the distributions in $S$. This may fail for the component $\LL_0$.

If we add to $S$ all pairwise $n$-th products $\tu^i (z)_{(n)} \tu^k (z)$ with $n \geq 0$, we will get a finite family 
${\overline S}$ of homogeneous distributions whose coefficients will span 
$\LL_j$ with $j\neq 0$, as well as $[\LL_j, \LL_{-j}]$ with $j\neq 0$ (see (\ref{kprod}) below).
Since $\LL$ is pluperfect, the coefficients of distributions in ${\overline S}$ will span all of $\LL$.
By Dong's Lemma (\cite{K}, Lemma 3.2), the distributions in $\barS$ are mutually local.
\end{proof}

\section{Regular Lie superalgebras of formal distributions and twisted forms of superconformal algebras}
\label{dloc}

We begin with a definition of a regular Lie superalgebra of formal distributions in a twisted setting.
Here we are following the ideas of \cite{K} and \cite{K2}.

\begin{definition}
\label{defloc}
(cf. \cite{K})
Let $\LL$ be a Lie superalgebra which is graded by the group $\C/\Z$
$$ \LL = \mathop\oplus\limits_{\alpha \in \C/\Z} \LL_{\alpha},$$
and let $T$ be an even derivation of $\LL$ preserving this grading. $\LL$ is called 
a regular Lie superalgebras of formal distributions if there exists a spanning family 
$\FF = \{ x^j (z) = \sum\limits_{n \in \alpha_j} x^j_n z^{-n-1} \, | \, x^j_n \in \LL_{\alpha_j} \}$
of $\LL$-valued mutually local formal distributions such that
every distribution in $\FF$ is $T$-covariant:
%
%
$$T(x^j(z)) = \sum\limits_{n\in\alpha_j} T(x^j_n) z^{-n-1} = \dd{z} x^j(z) .$$
\end{definition}

Lie bracket of two mutually local distributions $x(z)$, $y(w)$, supported on cosets 
$\alpha$ and $\beta$ respectively, can be expanded as follows (for details see \cite{K2}, Section 7):
\begin{equation}
\label{kprod}
\left[ x(z), y(w) \right] 
= \sum_{k=0}^N \frac{1}{k!} c^{(k)} (w) \left( \dd{w} \right)^k \delta_\alpha (w-z), 
\end{equation}
where distributions $c^{(k)} (w)$ are supported on $\alpha+\beta$ and 
$\delta_\alpha (w-z)$ is defined as
$$\delta_\alpha (w-z) = z^{-1} \sum_{n \in \alpha} \left( \frac{w}{z} \right)^n .$$
For $k \geq 0$ the $k$-th product $x(z)_{(k)} y(z)$ of distributions  $x(z)$ and $y(z)$ is defined as the distribution
$c^{(k)}(z)$ appearing in (\ref{kprod}).

Let $(\LL, T, \FF)$ be as in Definition \ref{defloc} and
let vector space $\bF$ be the linear closure of $\FF$ with respect to operations $\dd{z}$ and
$k$-th products ($k \geq 0$). Denote by $\bF_\alpha$ the subspace of distributions supported on a coset $\alpha \in \C/\Z$. We have a grading on $\bF$:
$$\bF =  \mathop\oplus\limits_{\alpha\in \C/\Z} \bF_\alpha.$$

By Dong's Lemma (\cite{K},  Lemma 3.2) distributions in $\bF$ are mutually local. It is also easy to check that they are $T$-covariant.

The space $\bF$ has a structure of a (Lie) conformal algebra. We refer to \cite{K} for the definition and details. We will denote this conformal algebra by $\CC$ and denote by $x^j$ the generators of $\CC$, corresponding to distributions $x^j (z)$. The operator on $\CC$ corresponding to differentiation $\dd{z}$ on $\bF$ will be denoted by $\del$.

 A grading on $\LL$ by $\C/\Z$ induces an automorphism $\theta$ of $\CC$, where 
$\theta (x^j) = e^{2\pi i \alpha_j} x^j$ and $\theta(\del x) = \del \theta(x)$ for all 
$x \in \CC$.

This gives us a functor $\Con$ from the category of regular Lie superalgebras of formal distributions to the category of conformal algebras decorated with a diagonalizable automorphism,
$(\CC, \theta) = \Con(\LL)$. Sometimes with a slight abuse of notations we will write $\CC = \Con(\LL)$.

Following \cite{K}, let us also construct functor $\Alg$ that goes in the opposite direction.

In addition to $k$-th products and differentiation, there is another operation on formal distributions that preserves locality -- it is multiplication by elements of $\C[z,z^{-1}]$.
This leads to the following notion:
\begin{definition}
An {\it affine} conformal algebra is a conformal algebra $\CC$ with an additional structure of a $\C[t,t^{-1}]$-module, satisfying the axioms:
$$\del(f x) = \frac{df}{dt} x + f \del(x) \leqno{(A1)}$$
and
$$(f x)_{(k)}(g y) = \sum\limits_{j=0}^\infty \frac{1}{j!} \frac{d^j f}{dt^j} g (x_{(k+j)} y),
\leqno{(A2)}$$
where $f, g \in \C[t,t^{-1}]$, $x, y \in \CC$.
\end{definition}

Given a conformal algebra $\CC$ and its diagonalizable automorphism $\theta$
(equivalently: grading by $\C/\Z \cong \C^*$), 
we can construct an affine conformal algebra $\tC$ via the process of {\it twisted 
affinization}:
$$\tC = \mathop\oplus\limits_{\alpha\in\C/\Z} t^{\alpha} \C[t,t^{-1}] \otimes \CC_\alpha,$$
where $\del$ and $k$-th products are defined via (A1) and (A2).

The space $\del \tC$ is an ideal in $\tC$ with respect to $0$-th product, and the quotient
$\tL = \tC / \del \tC$ is a Lie superalgebra $\Alg(\CC, \theta)$ associated to such a pair
with the Lie bracket given by $0$-th product (see \cite{K} for details). 
Its spanning family of mutually local distributions is 
$$\FF = \mathop\cup\limits_{\alpha\in\C/\Z} \bigg\{ x(z) = \sum_{n\in\alpha} (t^n x) z^{-n-1} \,\, \big| \,\,
x \in\CC_\alpha \bigg\}.$$
These distributions are covariant with respect to derivation $T =  - \frac{d}{dt}$. 


A conformal algebra is called {\it finite} if it is finitely generated as a $\C[\del]$-module.

\begin{theorem}
\label{finloc}
(\cite{K}, Corollary 4.7, Remark 2.7b)
Let $\LL$ be a regular Lie superalgebra of formal distributions with a finite spanning family $\FF$.
Then conformal superalgebra $\CC = \Con (\LL)$ is finite. Conversely, if $\CC$ is a finite conformal superalgebra and $\theta$ is its diagonalizable automorphism then $\Alg(\CC, \theta)$ is a regular Lie superalgebra of formal distributions with a finite spanning family,
\end{theorem}

Let us mention a relation between functors $\Con$ and $\Alg$.
Using the same proof as in \cite{K}, where untwisted case is treated, one obtains
\begin{theorem}
Let $\CC$ be a conformal algebra with a diagonalizable automorphism $\theta$. Then
$$\Con (\Alg(\CC, \theta)) \cong (\CC, \theta).$$
\end{theorem}

A counterpart of this statement is the following
\begin{theorem}
\label{AlgCon}
Let $\LL$ be a regular Lie superalgebra of formal distributions with a spanning family $\FF$. Then $\tL = \Alg(\Con(\LL))$ is a central extension of $\LL$. 

The homomorphism
$\sigma: \ \tL \rightarrow \LL$ given by $\sigma(t^n x) = x_n$ for $x(z) \in \bF_\alpha$, $n \in \alpha$.
\end{theorem}

Before we give the proof, let us illustrate this theorem with a simple example. 

\begin{example}
\normalfont
Let $\LL$ be a Heisenberg Lie algebra with a basis 
\break
$\left\{ x_n, \, c \,|\, n \in \Z \right\}$ and the Lie bracket
$[x_n, x_k] = n \delta_{n, -k} c$. Then $\LL$ is a regular Lie algebra with the spanning 
distributions $x(z) = \sum\limits_{n \in \Z} x_n z^{-n-1}$, $c(z) = c z^0$ 
and the derivation $T(x_n) = -n x_{n-1}$, $T(c) = 0$.
Let us take a regular subalgebra $\LL_1 \subset \LL$ with the spanning distributions
$y(z) = \left(\frac{d}{dz}\right)^2 x(z)$ and $c(z)$. Then $\LL_1$ has basis
$\left\{ x_n, c \,|\, n \in \Z\backslash \{ -2, -1 \} \right\}.$
It is easy to calculate that $\tL_1 = \Alg(\Con(\LL_1))$ is a Lie algebra 
with a basis  $\left\{ y_n, c \,|\, n \in \Z \right\}$ and Lie bracket
$[y_n, y_k] = 5! {n \choose 5} \delta_{n+k, 4} c$
and the derivation $T(y_n) = -n y_{n-1}$.
The homomorphism 
$\tL_1 \rightarrow \LL_1$ is given by $y_n \mapsto n (n-1) x_{n-2}$. Its kernel is a 
central ideal spanned by $y_0$, $y_1$. 
\end{example}

 The proof of Theorem \ref{AlgCon} is based on the following Lemma:
 \begin{lemma}
 \label{zshifts}
 Let $S$ be a finite subset in $\C$. Assume
 \begin{equation}
 \label{zdep}
 \sum_{n \in S} z^{n} x^n (z) = 0,
 \end{equation}
 where $\{ x^n (z) \, | \, n \in S \}$ is a family of $T$-covariant distributions in $\LL$.
 Then $x^n (z) = 0$ for every $n \in S$.
 \end{lemma}
 \begin{proof}
 Apply $zT$ to (\ref{zdep}):
 \begin{align*}
 0 &= \sum_{n \in S} z^{n+1} T(x^n (z)) =  \sum_{n \in S} z^{n+1} \frac{d}{dz} x^n (z) \\
 &=  z \frac{d}{dz} \sum_{n \in S} z^n x^n (z) - \sum_{n \in S}  z \frac{d}{dz} (z^n) x^n (z) 
 = - \sum_{n \in S} n z^{n} x^n (z) = 0.
 \end{align*}
Iterating these steps and applying Vandermonde determinant argument, we get the claim of the Lemma.
\end{proof}
 
 Now we can give a proof of Theorem \ref{AlgCon}.
 
 \begin{proof}
 It was proved in Section 2.7 of \cite{K} that $\sigma$ is a homomorphism of Lie algebras. Let us show that $\Ker \sigma$ 
 is central in $\tL$. Suppose $\sum\limits_{n \in S} t^n x^n \in \Ker \sigma$ for some finite family $\{ x^n \, | \, n \in S\} \subset \bF$. Then  $\sum\limits_{n \in S} x^n_n = 0$ in $\LL$. Taking the commutator with an arbitrary distribution $y(w) \in \bF$
 and applying (\ref{kprod}), we get
 $$ \sum_{n \in S} \sum_{j \geq 0} {n \choose j} w^{n-j} (x^n_{(j)} y) (w) = 0.$$
 Let us formally assume $x^n = 0$ for $n \not\in S$. Making a change of indices $m = n - j$, we get
 $$ \sum_{m} \sum_{j \geq 0} {m+j \choose j} w^{m} (x^{m+j}_{(j)} y) (w) = 0.$$
 By Lemma \ref{zshifts} we conclude that for every $m$
 \begin{equation}
 \label{zerosum}
 \sum\limits_{j \geq 0} {m+j \choose j} x^{m+j}_{(j)} y = 0
 \end{equation}
 in $\Con(\LL)$.
 
 Now consider the commutator in $\tL$:
 \begin{align*}
 \left[ \sum_{n \in S} t^n x^n, t^k y \right] = 
 \sum_{n \in S}  \sum_{j \geq 0} {n \choose j} t^{n+k-j} x^n_{(j)} y.
 \end{align*}
 Making the change of indices $m = n - j$, and applying (\ref{zerosum}), we obtain
 \begin{align*}
 \sum_{m} t^{m+k}  \sum_{j \geq 0} {m+j \choose j} x^{m+j}_{(j)} y = 0.
 \end{align*}
 Hence, every element of $\Ker \sigma$ is central in $\tL$. 
 \end{proof}
 
 Let us conclude this section by giving an example of a (non-simple) superconformal algebra which is a graded Lie superalgebra 
 of formal distributions, but not regular. As a result, one can associate with this superconformal algebra
 a finitely generated affine conformal superalgebra, but not a finite one.
 
 \begin{example}
\label{nonfinite} 
\normalfont
Let $\alpha \in \C$ and
let $\LL$ be a Lie algebra with a basis $\{ L_n, U_n, V_n \, | \, n \in \Z \}$ with Lie brackets
\begin{align*}
&[L_n, L_m] = (n-m) L_{n+m}, \\
&[L_n, U_m] = - (m + \alpha n) U_{n+m}, \\
&[L_n, V_m] = - (m + \alpha n) V_{n+m} + n U_{n+m}, \\
&[U_n, U_m] = [U_n, V_m] = [V_n, V_m] = 0.
\end{align*}
\end{example}

\section{Quasi-Poisson algebras}

In this section we will discuss a class of superconformal algebras that does not include all superconformal algebras, but 
encompasses all known simple superconformal algebras.

\begin{definition}
Let $A$ be a $\Z_2$-graded space with two bilinear operations $\cdot$ and $\{ , \}$ and with an even endomorphism
$P \in \End(A)$. We call $A$ a quasi-Poisson superalgebra (or QP-superalgebra, for short) if the following axioms hold:

(QP1) $(A, \cdot)$ is a commutative associative unital superalgebra,

(QP2) $(A, \{ , \})$ is a Lie superalgebra,

(QP3) for all $a, b \in A$
$$ P(a \cdot b) = P(a) \cdot b  + a \cdot P(b) - a \cdot b,$$

(QP4) for all $a, b, c \in A$
\begin{align*}
 a \cdot P( \{ b, c \}) &= \{ a \cdot P(b), c \} + (-1)^{\pp{a}\pp{b}} \{ b, a \cdot P(c) \} \\
 &+ \{ a, b \} \cdot P(c) - (-1)^{\pp{a} \pp{b}} P(b) \cdot \{a, c\}.
 \end{align*}
 
\end{definition}

 Let us point out some elementary consequences of these axioms. 
 
 Axiom (QP3) is equivalent to the statement that operator $\DD = P - \Id$ is a derivation
 of the commutative superalgebra $(A, \cdot)$. This implies $P(1) = 1$. 
 
 Setting $a=1$ in (QP4) we get
 \begin{equation}
\label{Pbracket}
 P( \{ b, c \}) = \{ P(b), c \} + \{ b, P(c) \} 
 + \{ 1, b \} \cdot P(c) - P(b) \cdot \{1, c\} .
 \end{equation}
 
 Further setting $b = 1$ we get
 \begin{equation}
 \label{Pad}
 P ( \{ 1, c \} ) = \{ 1, P(c) \},
 \end{equation}
which means that $P$ commutes with $\ad (1)$.

 Setting $c=1$ in (QP4) and taking into account (\ref{Pad}), we get
 \begin{equation}
 \label{adiff}
 \{ 1, a \cdot P(b) \} = \{ 1, a \} \cdot P(b) + a \cdot \{ 1, P(b) \}.
 \end{equation}
 

Let us mention three natural classes of quasi-Poisson algebras. 

\begin{lemma}
\label{poisson}
Let $A$ be a Poisson superalgebra with an even derivation $\DD$ of $(A, \cdot)$ such that $P = \DD + \Id$ is a derivation of
$(A, \{,\})$. Then $A$ is a quasi-Poisson superalgebra.
\end{lemma}

\begin{lemma}
\label{AD}
Let $(A, \cdot)$ be a commutative unital superalgebra with an even derivation $\DD$. Define a Lie bracket $\{,\}$ on $A$ via
$$ \{ a , b \} = a \cdot \DD(b) - \DD(a) \cdot b .$$ 
Then $A$ is a quasi-Poisson superalgebra with $P = \DD + \Id$. 
\end{lemma}

\begin{lemma}
\label{QPA}
Let $(A, \cdot)$ be a commutative unital superalgebra with two even operators $P, Q \in \End(A)$, such that $PQ = QP$,
$P$ satisfies (QP3), $Q$ satisfies
\begin{equation}
\label{Qdiff}
Q(a \cdot P(b)) = Q(a) \cdot P(b) + a \cdot PQ(b).
\end{equation}
Then operation
\begin{equation}
\label{newbrak}
\br{a}{b} = Q(a) \cdot P(b) - P(a) \cdot Q(b)
\end{equation}
is a Lie bracket, and $(A, \cdot, \br{}{})$ is a QP-superalgebra.
\end{lemma}

Verification of the above lemmas is straightforward. Note that quasi-Poisson superalgebra constructed in Lemma \ref{AD} is a special case of QP-superalgebra of Lemma \ref{QPA}, where we take $P = \DD + \Id$ and $Q = \DD$.

It turns out that a quasi-Poisson superalgebra may have more than one QP-structure. First, let us answer the following simple question: given two Lie brackets $\{ , \}$ and $\br{}{}$ on the same vector space, when is their arbitrary linear combination also a Lie bracket?

\begin{lemma}
\label{mLie}
Suppose $\{ , \}$ and $\br{}{}$ are two Lie brackets on the same vector space $A$. For two scalars $s_1, s_2 \in \C$ 
consider a new bilinear operation on $A$:
$$ [a, b] = s_1 \{ a, b \} + s_2 \br{a}{b} .$$
Then $[ , ]$ is a Lie bracket for all values of $s_1, s_2$ if and only if the following identity holds:
$$ \{ \br{a}{b} , c \} + \br{ \{ a, b \} }{c} + \text{ cyclic permutations } = 0.$$
\end{lemma}

The proof is elementary and we omit it. We shall call two Lie brackets satisfying the condition of Lemma \ref{mLie} compatible.

\begin{proposition}
\label{doubleQP}
Let $(A, \cdot, \{,\}, P)$ be a quasi-Poisson algebra. Consider operator $Q = \ad(1)$ and a new Lie bracket $\br{}{}$ given by
(\ref{newbrak}). Then Lie brackets $\{,\}$ and $\br{}{}$ are compatible and $A$ is a quasi-Poisson algebra with respect to an arbitrary linear combination of these two Lie brackets.
\end{proposition}
\begin{proof}
First of all, we point out that $Q$ satisfies the conditions of Lemma \ref{QPA} by (\ref{Pad}) and (\ref{adiff}). Let us show that 
 Lie brackets $\{,\}$ and $\br{}{}$ are compatible. We have
 \begin{equation*}
  \{ \br{a}{b} , c \} = \{ Q(a) \cdot P(b), c \} - \{ P(a) \cdot Q(b) \}. 
 \end{equation*}
 and
 \begin{equation*}
 \br{ \{ a, b \} }{c} = \{ Q(a), b \} \cdot P(c) + \{ a, Q(b) \} \cdot P(c) - P\left( \{ a, b \} \right) \cdot Q(c) .
 \end{equation*}
 Adding together cyclic permutations of two previous expressions and applying (QP4), we get zero, hence the two Lie brackets are compatible. Since axiom (QP4) contains a single Lie bracket in each of its terms, it holds for any linear combination of these two Lie brackets.
\end{proof}

Let us now construct a functor from the category of quasi-Poisson superalgebras to the category of Lie superalgebras.

\begin{theorem}
\label{mainQP}
Let $A$ be a quasi-Poisson superalgebra. Set $\LL(A) = \C [t, t^{-1}] \otimes A$ and for $a \in A$ denote $t^n \otimes a$ by
$a_n$. Define a bilinear multiplication $[ , ]$ on $\LL(A)$ via
\begin{equation}
\label{Lie}
[a_n, b_m ]   = n ( a \cdot P(b) )_{n+m} - m ( P(a) \cdot b )_{n+m} + \{ a, b \}_{n+m} .
\end{equation}  
Then $( \LL(A), [ , ] )$ is a Lie superalgebra.
\end{theorem}

The proof of this theorem is straightforward. Anticommutativity of $[ , ]$ follows from the commutativity of $\cdot$ and 
anticommutativity of $\{ , \}$. Jacobi axiom for $[ , ]$ follows from associativity of $\cdot$, Jacobi axiom for $\{ , \}$,
(QP3) and (QP4).

One can notice a resemblance of (\ref{Lie}) to the formula for the contact brackets given in Section 0.3 of \cite{GSL}.

Note that in general we can not completely recover the quasi-Poisson structure of $A$ from the Lie bracket on $\LL(A)$ since we 
can not use (\ref{Lie}) to define $\cdot$ product of two elements of $\Ker P$.

The construction we present here follows the idea of \cite{BN} which introduced Novikov algebras in a similar context
(see also \cite{GD} and \cite{OZ} for details on Novikov algebras). If $(A, \cdot)$ is a commutative superalgebra and $P = \DD + \lambda \Id$, where $\DD$ is a derivation of $(A, \cdot)$ and $\lambda \in \C$  then $A$ becomes a Novikov superalgebra with respect to a new product $\circ$ defined as $a \circ b = P(a) \cdot b$.

 If $E$ is an even idempotent of a quasi-Poisson superalgebra $A$, i.e., $E \cdot E = E$, and $P(E) = \omega E$ with $\omega \neq 0$, then the family $\{ E_n | n \in \Z \}$ spans a Virasoro subalgebra $\Vir$ in $\LL(A)$. For $E =1$ we shall denote the corresponding elements by $L_n$. 
 
We shall call an element $a \in A$ {\it primary} if it is a common eigenvector for $P$ and $Q = \ad(1)$. Suppose 
\begin{equation}
\label{eigen}
Q(a) = \chi a, \ P(a) = \omega a.
\end{equation} 
Then
$$[ L_n, a_m ] = (n \omega - m + \chi)  a_{n+m}.$$
We see that the family $\{ a_m \}$ spans in $\LL(A)$ a tensor $\Vir$-module 
\break
$V(-\omega, -\chi + \Z)$.  

\begin{lemma}
\label{prim}
Let $a_1$, $a_2$ be two primary elements in a quasi-Poisson superalgebra $A$ with $Q(a_i) = \chi_i a_i$, $P(a_i) = \omega_i a_i$, $i=1, 2$. Then

(a) $a_1 \cdot P(a_2)$ is a primary element of $A$ with the eigenvalues $\chi_1 + \chi_2$ and $\omega_1 + \omega_2 - 1$ for $Q$ and $P$ respectively.

(b) $\{ a_1, a_2 \} + \br{a_1}{a_2}$ is a primary element of $A$ with the eigenvalues $\chi_1 + \chi_2$ and $\omega_1 + \omega_2$ for $Q$ and $P$ respectively.
\end{lemma}
\begin{proof}
The claim of part (a) follows from (\ref{adiff}) and (QP3). Let us show part (b). Since $Q = \ad(1)$, it is a derivation of 
$\{ , \}$. It follows from (\ref{adiff}) that it is also a derivation with respect to $\br{}{}$. This establishes the claim for the eigenvalue of $Q$. For the action of $P$ we apply (\ref{Pbracket}) and (QP3):
$$ P \{ a_1, a_2 \}  = \{ P(a_1), a_2 \} + \{ a_1, P(a_2) \} + \br{a_1}{a_2},$$
$$ P \br{a_1}{a_2} = \br{P(a_1)}{a_2} + \br{a_1}{P(a_2)} - \br{a_1}{a_2},$$
and the claim follows.
\end{proof}
 
 \begin{theorem}
 \label{QPnprod}
Let $A$ be a finite-dimensional quasi-Poisson superalgebra.
 Fix a Virasoro subalgebra in $\LL(A)$, which is generated by $1 \in A$. 
 
(a) $\LL(A)$ is a superconformal algebra if and only if $Q = \ad (1)$ is diagonalizable on $A$. 

Let us assume that both $P$ and $Q$ are diagonalizable on $A$.
 
(b) Set $T = \ad (L_{-1})$. Then $\LL(A)$ is a regular superconformal algebra.

(c) We have $\Con (\LL(A)) = (\CC, \psi)$, where $\CC$ is a conformal superalgebra generated by $A$
as a free $\C[\del]$-module, $\CC \cong \C[\del] \otimes A$,  
with the following $n$-th products on the generators $a, b \in A$:
\begin{align}
&a_{(0)} b = \del (P(a) \cdot b) + \cb{a}{b} + \br{a}{b}, \nonumber \\
\label{nthprod}
&a_{(1)} b = P(a) \cdot b + a \cdot P(b),  \\
&a_{(n)} b = 0 \ \text{for \ } n \geq 2, \nonumber
\end{align}
and $\psi$ is a diagonal automorphism of $\CC$ which commutes with $\del$ and is defined on $A$ by 
$\psi = \exp(2 \pi i (P - Q))$.

(d) Every automorphism of a quasi-Poisson superalgebra $A$ extends to an automorphism of conformal superalgebra $\CC$ fixing the Virasoro element.
\end{theorem}
\begin{proof}
To establish (a), we need to show that $\LL(A)$ has a grading by the the eigenspaces of $\ad (L_0)$.
We get from (\ref{Lie}) that
$$ [L_0, a_m ] = -m a_m + Q(a)_m ,$$
from which the claim (a) follows.

Note that $P$ and $Q$ commute, so if both of these operators are diagonalizable, then they are diagonalizable in the same basis. In this case $A$ has a basis of primary elements. Then as the adjoint $\Vir$-module, $\LL(A)$ decomposes into a direct sum of tensor modules. By Theorem \ref{grading}, $\LL(A)$ is a graded by the group $\C/\Z$, 
$$\LL(A) = \mathop\oplus\limits_{\alpha \in \C/\Z} \LL(A)_\alpha,$$
where for each primary element $a \in A$ satisfying (\ref{eigen}), the family $\{ a_m | m \in \Z \}$ belongs to component
corresponding to the coset $\alpha = \omega - \chi  + \Z$. Let us construct a primary field corresponding to this family:
$a(z) = \sum_{m \in \Z} a_m z^{-m+\chi-\omega-1}$. 
Using
$$ [L_{-1} , a_m] = (-m - \omega + \chi) a_{m-1},$$
it is straightforward to check that $a(z)$ is $T$-covariant,
$[L_{-1}, a(z)] = \dd{z} a(z)$. Mutual locality of such fields follows from the expression (\ref{Lie}) 
for the Lie bracket in $\LL(A)$:
\begin{align}
\left[ a(z), b(w) \right] =& \left( \dd{w} (P(a) \cdot b)(w) \right) \delta_\alpha (w - z) \nonumber \\
&+ \left( P(a) \cdot b + a \cdot P(b) \right) (w) \dd{w} \delta_\alpha (w - z) \label{delta} \\
&+ \left( \{ a, b \} + \br{a}{b} \right) (w) \delta_\alpha (w - z), \nonumber
\end{align}
where $a \in A_\alpha$, $b \in A_\beta$.

It follows from Lemma \ref{prim} that $\exp(2\pi i P)$ and $\exp(\lambda Q)$ with $\lambda \in \C$ are commuting 
automorphisms of $\CC$, and so is $\psi = \exp(2 \pi i (P - Q))$.

Finally, claim (d) follows from (\ref{delta}).
\end{proof}

We call a quasi-Poisson algebra {\it reduced} if $Q = \ad (1) = 0$. We can see from Proposition \ref{doubleQP} that starting
from an arbitrary QP-algebra $A$, we can construct a reduced QP-algebra by taking a new Lie bracket on the same space:
$$\bcb{a}{b} = \cb{a}{b} + \br{a}{b}.$$
We will use bold brackets for this new operation and denote the quasi-Poisson algebra $(A, P, \cdot, \bcb{}{})$ by $\dot{A}$. 
It follows from (\ref{Pbracket}) that in a reduced QP-algebra operator $P$ is a derivation of the Lie bracket.
Theorem \ref{QPnprod} (c) implies that conformal algebras associated to $\LL(A)$ and $\LL(\dot{A})$ are isomorphic 
(assuming that $P$ and $Q$ are diagonalizable), hence $\LL(A)$ and $\LL(\dot{A})$ are twisted forms of each other.

We conclude this section with the following

{\bf Question.} 
{\it Does it exist a simple finite-dimensional quasi-Poisson superalgebra with a non-diagonalizable operator $P$?}

 Applying Theorem \ref{mainQP} to such an algebra might yield a simple superconformal algebra with no associated finite 
conformal superalgebra.

%
%
%
%
%

\section{Appendix: Simple quasi-Poisson algebras}

We will describe here simple quasi-Poisson algebras corresponding to known simple superconformal algebras.
Before we do this, let us introduce some notations.

Let $\Lambda(N)$ be the Grassmann algebra, which is the exterior algebra with odd generators $\xi_1, \ldots, \xi_N$.
This algebra is associative and supercommutative. 

Let $\del_i = \dd{\xi_i}$ be the derivation in variable $\xi_i$, $i=1, \ldots,N$.
If $f \in \Lambda(N)$ does not contain the variable $\xi_i$ then $\del_i (\xi_i f) = f$ and $\del_i (f) = 0$.

Let $W(N)$ be the Lie superalgebra of derivations of $\Lambda(N)$. It is a free $\Lambda(N)$-module of rank $N$:
$$W(N) = \mathop\oplus\limits_{i=1}^N \Lambda(N) \del_i .$$

We shall denote the supercommutative product in $\Lambda(N)$ and the left action of $\Lambda(N)$ on $W(N)$ by
concatenation: $fg \in \Lambda(N)$, $f \eta \in W(N)$ for $f, g \in \Lambda(N)$, $\eta \in W(N)$.

In addition to the usual (left) derivations $\del_i$, we introduce {\it right} derivations $\sdel_i$, which are defined by
$$(f \xi_i)\sdel_i = f, \ \ f \sdel_i = 0, $$
where $f\in \Lambda(N)$ does not contain the variable $\xi_i$. There is an obvious relation between left and right differentiations:
$$ \del_i (f) = (-1)^{\pp{f} - 1} (f)\sdel_i ,$$
and the Leibniz rule is written as:
$$ (fg)\sdel_i = f (g)\sdel_i + (-1)^{\pp{g}} (f) \sdel_i g .$$
The use of right derivations allows us to avoid non-intuitive negative signs in some formulas.

Consider also the ring 
$\R_N = \C[t,t^{-1}] \otimes \Lambda (N)$ and its Lie superalgebra of derivations
$$W(1, N) = \R_N \delt \oplus \sum\limits_{i=1}^N \R_N \del_i ,$$
where $\delt = \dd{t}$. Lie superalgebra $W(1, N)$ is a simple superconformal algebra,
and in fact, all known simple superconformal algebras may be realized as subalgebras in $W(1, N)$.

\
  
{\bf 1. Contact type $K_N$.}
Quasi-Poisson algebra $K_N$ is actually a Poisson algebra. Its space is the Grassmann algebra
$\Lambda(N)$ with $\cdot$ being the usual supercommutative product in $\Lambda(N)$.
The Lie bracket $\cb{}{}$ is the Poisson bracket on $\Lambda(N)$ defined as follows:
$$\cb{f}{g} = \sum_{i=1}^N (f)\sdel_i \, \del_i (g) .$$
Operator $P$ is defined by
$$ P = \Id - \frac{1}{2} \sum_{i=1}^N \xi_i \del_i .$$
It is easy to check that $P$ satisfies (QP2) and is a derivation of $\cb{}{}$. By Lemma \ref{poisson}, $K_N$ is a quasi-Poisson
algebra.

Consider differential 1-form
$$\omega = dt - \sum_{i=1}^N \xi_i d \xi_i .$$

 Superconformal algebra $\LL(K_N) \cong K(1, N)$ is the following subalgebra in $W(1, N)$:
$$ K(1, N) = \left\{ \eta \in W(1, N) \,  | \, \eta \omega = f \omega \text{ for some } f\in\R_N \right\} .$$ 
Superconformal algebras $K(1, N)$ are simple, except in the case of $N=4$. Superalgebra $K(1,4)$ has a simple ideal of codimension 1.

\

{\bf 2. Witt type $W_N$.}

Quasi-Poisson algebra $W_N$ corresponds to superconformal algebra $W(1, N)$. Its underlying space
is $\Lambda(N) \oplus W(N)$. The supercommutative product $\cdot$ in $W_N$ 
is given by
$$ f \cdot g = fg, \ \ f \cdot \eta = f \eta, \ \ \eta \cdot \tau = 0,$$
for $f, g \in \Lambda(N)$, $\eta, \tau \in W(N)$.
The Lie bracket $\cb{}{}$ is given as follows:
$$\cb{f}{g} = 0, \ \ \cb{\eta}{f} = \eta(f), \ \ \cb{\eta}{\tau} = [\eta, \tau] .$$
Operator $P$ is defined by $P(f) = f$, $P(\eta) = 0$. This is a simple quasi-Poisson superalgebra, which is not a Poisson superalgebra.

Elements in $W_N$ give rise to the following families in $\LL(W_N) \cong W(1, N)$:
$$ f_n = - t^{n+1} f \delt, \ \ \eta_n = t^n \eta .$$

\

{\bf 3. Divergence zero type $S_{N, \alpha}$, $N \geq 2$.}

The divergence operator on $W(1, N)$ is defined as
$$\dv \left( P_0 \delt + \sum_{i=1}^N P_i \del_i \right) 
= \pd{P_0}{t} -  \sum_{i=1}^N (P_i) \sdel_i .$$
Negative sign in the above formula is placed in order to agree with the notations of \cite{KL}.

It follows from the equality
$$\dv ([\eta, \tau]) = \eta( \dv \tau) - (-1)^{\pp{\eta} \pp{\tau}} \tau( \dv \eta) $$
that the subspace $S(1, N) \subset W(1, N)$ of divergence zero elements forms a subalgebra. 
In the same way we define the divergence zero subalgebra $S(N)$ in $W(N)$.

It turns out that one can generalize a construction of $S(1, N)$ and get a family of superconformal algebras.

Fix even element $F \in \R_N$. Consider the following subspace in $W(1, N)$:
$$ S_F (1, N) = \left\{ \eta \in W(1, N) \, | \, \dv (F \eta) = 0 \right\} .$$
 Since
 \begin{align*}
\dv F [\eta, \tau] =& \eta( \dv(F\tau)) - (-1)^{\pp{\eta}\pp{\tau}} \tau( \dv(F \eta)) \\
& - \dv( F \eta) \,\dv \tau +   (-1)^{\pp{\eta}\pp{\tau}} \dv( F \tau) \,\dv \eta,
\end{align*}
$S_F (1, N)$ is actually a Lie subalgebra in $W(1, N)$. If $F$ is invertible then
$S_F (1, N) = F^{-1} S(1, N)$.

Now set $F = t^\alpha$, where we allow $\alpha$ to be in $\C$.
Then $S_\alpha (1, N)$ is 
$$ S_\alpha (1, N) = \left\{ \eta \in W(1, N) \, | \, \dv (t^\alpha \eta) = 0 \right\} .$$ 

Consider also derivation 
$$ d = - \frac{1}{N} \sum_{i=1}^N \xi_i \del_i .$$

Lie superalgebra $S_\alpha (1, N)$ is a superconformal algebra with the Virasoro subalgebra spanned by
$$L_n = -t^{n+1} \del_t + (n+\alpha+1) t^n d .$$
The derived subalgebra $S^\prime_\alpha (1, N)$ is a simple superconformal algebra.
If $\alpha \not\in\Z$, the derived subalgebra coincides with $S_\alpha (1, N)$, and if $\alpha\in\Z$,  it has codimension 1 in  and
$S_\alpha(1, N) = S^\prime_\alpha (1, N) \oplus \C t^{-\alpha} \xi_1 \ldots \xi_N \del_t$. 

Consider linear functional 
$$\int: \, \Lambda(N) \rightarrow \C, $$
where $\int f$ is the coefficient of the monomial $\xi_1 \ldots \xi_N$ in $f$.
Denote by $\Ls (N)$ the kernel of $\int$, it is the subspace in $\Lambda(N)$ spanned by monomials of degree less than $N$. 
The obvious equality $\int \del_i f = 0$ implies the integration by parts formula:
$$ \int f \del_i (g) = \int (f)\sdel_i g.$$
We can immediately see from the integration by parts formula that for $\eta \in W(N)$, $f \in \Lambda(N)$
$$ \int \eta(g) = - \int \dv(\eta) g .$$
As a consequence, we conclude that for $\eta \in S(N)$, we have $\eta (f) \in \Ls (N)$.

The space of the quasi-Poisson superalgebra $S_{N, \alpha}$ is $\Ls (N) \oplus S(N)$. We define operator $P$ on $S_N$ by
\begin{equation}
\label{SP}
P (f) = f + d(f), \ \ P(\eta) = [d, \eta] .
\end{equation}
We point out that even though $d$ does not belong to $S(N)$, it still acts on $S(N)$ as an outer derivation.
It is easy to see that $P$ is invertible on $\Ls (N)$.

Supercommutative multiplication $\cdot$ on $S_{N, \alpha}$ is defined by
\begin{equation}
\label{commult}
f \cdot g = fg, \ \ f \cdot \eta = f \eta - (-1)^{\pp{f} \pp{\eta}} P^{-1} (\eta(f)) d, \ \ \eta \cdot \tau = 0.
\end{equation}

In order to see that $f \cdot \eta \in S(N)$ for $\eta \in S(N)$, we need to apply the following two formulas:
$$ \dv (f \eta) = f \dv \eta + (-1)^{\pp{f} \pp{\eta}} \eta(f) ,
\ \ \  \dv ( g d ) = P(g) .$$
Lie bracket in $S_{N, \alpha}$ is given by
\begin{align*}
&\cb{f}{g} = (\alpha + 1) (f d(g) - d(f) g), \\
&\cb{f}{\eta} =  (\alpha + 1) \left(f \cdot [d, \eta] \right) - (-1)^{\pp{f}\pp{\eta}} P^{-1} (\eta (P (f))), \\
&\cb{\eta}{\tau} = [\eta, \tau].
\end{align*}

We have $\LL(S_{N, \alpha}) \cong S^\prime_\alpha (1, N)$.

Elements of $S_{N, \alpha}$ generate the following families in $S^\prime_\alpha (1, N)$:
\begin{align*}
&f_n = - t^{n+1} P(f) \del_t + (n + \alpha + 1) f d, \\
&\eta_n = t^n \eta .
\end{align*}

Note that a reduced QP-superalgebra ${\dot S}_{N, \alpha}$ is the same for all $\alpha$ and has a Lie bracket
\begin{align*}
\bcb{f}{g} = 0, \ \
\bcb{\eta}{f} = P^{-1} (\eta (P (f))), \ \
\bcb{\eta}{\tau} = [\eta, \tau].
\end{align*}
Hence superconformal algebras $S^\prime_\alpha (1, N)$ correspond to different automorphisms of the same conformal 
superalgebra.

\

{\bf 4. Type $\tS_N$, $N$ even.}

Let $N$ be even and take $F = 1 + \xi_1 \ldots \xi_N$. 
Consider the following subalgebra in $W(1, N)$:
$$ \tS (1, N) = \left\{ \eta \in W(1, N) \, | \, \dv (F \eta) = 0 \right\} .$$
It is a simple superconformal algebra, and the corresponding quasi-Poisson algebra  $\tS_N$ has the following 
structure. As a commutative algebra, $\tS_N$ is isomorphic to $S_N$: $\tS_N = \Lambda^\circ (N) \oplus S(N)$, 
multiplication given by  (\ref{commult}), and operator $P$ given by (\ref{SP}).

Lie bracket in $\tS_N$ is defined by
\begin{align*}
&\cb{f}{g} =f d(g) - d(f) g, \\
&\cb{f}{\eta} =  f \cdot [d, \eta] - (-1)^{\pp{f}\pp{\eta}} P^{-1} \big(\eta \big((1 -  \xi_1 \ldots \xi_N) P (f)\big)\big), \\
&\cb{\eta}{\tau} = (1 -  \xi_1 \ldots \xi_N) [\eta, \tau] - \eta( \xi_1 \ldots \xi_N ) \tau 
+ (-1)^{\pp{\eta}\pp{\tau}} \tau( \xi_1 \ldots \xi_N ) \eta .
\end{align*}

\

{\bf 5. Cheng-Kac algebra $\CK_6$.}

 We begin with the Poisson algebra $\Lambda(6)$ and we denote by $\Lambda{6 \choose k}$ the span of monomials of degree $k$, so that
$$\Lambda(6) = \mathop\oplus\limits_{k=0}^6 \Lambda{6 \choose k}.$$

We define the Hodge dual of a monomial $\xi_{i_1} \ldots \xi_{i_k}$ as 
$(\xi_{i_1} \ldots \xi_{i_k})^* = \pm \xi_{j_1} \ldots \xi_{j_{6-k}}$ with the property that
$\xi_{i_1} \ldots \xi_{i_k}  (\xi_{i_1} \ldots \xi_{i_k})^* = \xi_1 \xi_2 \ldots \xi_6$.

We consider a polarization on the middle space of the above decomposition 
$\Lambda{6 \choose 3} =  \Lambda^+{6 \choose 3} \oplus \Lambda^-{6 \choose 3}$, where 
the subspace $\Lambda^\pm {6 \choose 3}$ is spanned by the elements of the form
$\xi_i \xi_j \xi_k \pm \sqrt{-1} (\xi_i \xi_j \xi_k)^*$.

The space of the Cheng-Kac algebra is
$$\CK_6 = \sum_{k=0}^2 \Lambda{6 \choose k} \oplus \Lambda^+ {6 \choose 3}.$$

We are going to define a Lie bracket and a supercommutative multiplication on this space. It turns out that $\CK_6$ is closed under the Poisson bracket of $\Lambda(6)$, so it inherits Lie superalgebra structure. We construct a supercommutative product
$\cdot$ in two steps. First we consider the usual product on $\Lambda(6)$. The subspace
$\Lambda^- {6 \choose 3} \oplus  \sum_{k=4}^6 \Lambda{6 \choose k}$ is an ideal, and we get a supercommutative product
on $\CK_6$, identifying it as the quotient of $\Lambda(6)$ by this ideal. Next, we need to twist this product.

Consider an invertible endomorphism $\varphi$ on $\CK_6$, which is $2 \Id$ on $\Lambda^+ {6 \choose 3}$, and $\Id$ on
all other graded components. We define $\cdot$ on $\CK_6$ as a twisted product:
$$ a \cdot b = \varphi( \varphi^{-1}(a) \varphi^{-1}(b) ).$$

We also define operator $P$ on $\CK_6$ in the same way as in $\Lambda(6)$: $P = \Id - \frac{1}{2} \sum_{i=1}^6 \xi_i \del_i$.

It turns out that $(\CK_6, \cdot, \cb{}{}, P)$ is a simple quasi-Poisson algebra (but not a Poisson algebra).
Exceptional superconformal Cheng-Kac algebra $\CK(6)$ is a superconformal algebra associated with the quasi-Poisson algebra 
$\CK_6$. See \cite{CK} for the embedding of $\CK(6)$ into $K(1, 6)$.

\end{document}